\documentclass[11pt]{amsart}
\usepackage{amssymb, latexsym, mathrsfs, hyperref}

\numberwithin{equation}{section}
\newtheorem{Thm}[equation]{Theorem}

\theoremstyle{definition}

\begin{document}

\title [Automorphic correction of $E_{10}$]
{Automorphic correction \\ of the hyperbolic Kac-Moody algebra $E_{10}$}
\author[Henry Kim]{Henry H. Kim$^{\star}$}
\thanks{$^{\star}$ partially supported by an NSERC grant.}
\address{Department of
Mathematics, University of Toronto, Toronto, ON M5S 2E4, CANADA and Korea Institute for Advanced Study, Seoul, Korea}
\email{henrykim@math.toronto.edu}
\author[Kyu-Hwan Lee]{Kyu-Hwan Lee}
\address{Department of
Mathematics, University of Connecticut, Storrs, CT 06269, U.S.A.}
\email{khlee@math.uconn.edu}
\begin{abstract}
In this paper we study automorphic correction of the hyperbolic Kac-Moody algebra $E_{10}$, using the Borcherds product for $O(10,2)$ attached to a weakly holomorphic modular form of weight $-4$ for $SL_2(\mathbb Z)$. We also clarify some aspects of automorphic correction for Lorentzian Kac-Moody algebras and give heuristic reasons for the expectation that every Lorentzian Kac-Moody algebra has an automorphic correction. 
\end{abstract}

\maketitle

\section*{Introduction}

In his Wigner medal acceptance speech \cite{Kac-97}, V. Kac began with the remark, ``It is a well kept secret that the theory of Kac-Moody algebras has been a disaster." He continued to mention two exceptions: the affine Kac-Moody algebras and Borcherds' algebras. He explained the common feature of these algebras with the idea of locality. If we look at a narrower class of Borcherds' algebras related to the Monster Lie algebras \cite{Bor-90, Bor-92} and modular products \cite{Bor-95}, we observe another common feature of these algebras. Namely, the denominator functions of these algebras are automorphic forms.

One of Borcherds' motivations for constructing the fake Monster Lie algebra was to extend or {\em correct} the underlying Kac-Moody Lie algebra to obtain a generalized Kac-Moody algebra whose denominator function is an automorphic form. Gritsenko and Nikulin pursued Borcherds' idea to construct many examples of {\em automorphic correction} for rank $3$ Lorentzian Kac-Moody algebras \cite{GN-96, GN-97, GN-02}. For example, for Feingold and Frenkel's rank $3$ Kac-Moody algebra $\mathcal F$ \cite{FF}, the automorphic correction is given by the Siegel cusp form $\Delta_{35}(Z)$ of weight $35$, called {\em Igusa modular form}. See \cite{GN-96} or Section \ref{rank3} of this paper.

What are the benefits of having automorphic correction? First, automorphic forms satisfy a lot more of symmetries than the usual denominator function. For instance, we can prove generalized Macdonald identities as consequences of modular transforms \cite[Theorem 6.5]{Bor-95}. Second, we obtain generalized Kac-Moody algebras with known root multiplicities and may apply analytic tools such as the method of Hardy-Ramanujan-Rademacher to compute asymptotic formulas for root multiplicities, and thereby obtain bounds for root multiplicities of underlying Kac-Moody algebras. This analytic approach was taken in \cite{KimLee} to obtain upper bounds for root multiplicities of $\mathcal F$. Third, we may get connections to and applications for other branches of mathematics. Indeed, Borcherds'  work  \cite{Bor-96} on the fake Monster Lie superalgebra has applications to the moduli space of Enriques surfaces, and Gritsenko and Nikulin's work \cite{GN-96} is related to the moduli space of K3 surfaces.  

This paper is a continuation of our work on automorphic correction of hyperbolic Kac-Moody algebras. 
In our previous paper \cite{KimLee-1}, automorphic correction for some rank $2$ symmetric hyperbolic Kac-Moody algebras was constructed using Hilbert modular forms which are Borcherds lifts of weakly holomorphic modular forms \cite{Bor-98, BB}. In this paper, we clarify some aspects of automorphic correction for Lorentzian Kac-Moody algebras and give heuristic reasons for the expectation that every Lorentzian Kac-Moody algebra has an automorphic correction. 
We consider the examples of the rank $3$ hyperbolic Kac-Moody algebra $\mathcal F$ and rank $2$ symmetric hyperbolic Kac-Moody algebras, and explain why automorphic correction is highly non-trivial from the point of view of automorphic forms.

After that, we focus on the Kac-Moody algebra $E_{10}$. The algebra $E_{10}$ has attracted much attention from mathematical physicists, for example, \cite{BGGN, KN, DKN, BHK}. Moreover, S. Viswanath \cite{Vis} showed that 
it contains every simply laced hyperbolic Kac-Moody algebra as a Lie
subalgebra. The root multiplicities of this algebra was first studied by Kac, Moody and Wakimoto \cite{KMW}.

In Section \ref{hyperbolic}, we review a concrete realization of root lattice of $E_{10}$ as $2\times 2$ hermitian matrices over octavians, and in the next section, we obtain automorphic correction of the hyperbolic Kac-Moody algebra $E_{10}$. The automorphic correction is provided by a Borcherds lift on $O(10,2)$ of the weakly holomorphic modular form $f(\tau) = E_4(\tau)^2/\Delta_{12}(\tau)$, where $E_4(\tau)$ is the Eisenstein series of weight $4$ and 
$\Delta_{12}(\tau)$ is the modular discriminant function. Since $f(\tau)=\sum_n c(n)q^n$ has positive Fourier coefficients,
the automorphic correction $\mathcal G$ of $E_{10}$ is a generalized Kac-Moody algebra. However, $c(n)$ is too large to produce good upper bounds for root multiplicities of $E_{10}$. 
At the end of this paper, we give for comparison two other generalized Kac-Moody algebras which contain $E_{10}$. Their denominators are not automorphic forms, but they yield good upper bounds for root multiplicities of $E_{10}$. 

\subsection*{Acknowledgments} This joint work was initiated at the Korea Institute for Advanced Study (KIAS) in the summer of 2012, and was completed at the Institute for Computational and Experimental Research in Mathematics (ICERM) in the spring of 2013. We thank both institutes for stimulating environments for research.


\section{Automorphic Correction} \label{correction}

In this section, we recall the theory of automorphic correction established by Gritsenko and Nikulin \cite{GN-96,GN-97,GN-02}. The original idea of automorphic correction can be traced back to Borcherds' work \cite{Bor-90,Bor-92}.

\subsection{Modular forms on $O(n+1,2)$}

Let $(V,Q)$ be a non-degenerate quadratic space over $\mathbb Q$ of type $(n+1,2)$.
Let $V(\Bbb C)$ be the complexification of $V$ and $P(V(\Bbb C))=(V(\Bbb C)-\{0\})/\Bbb C^*$ be the corresponding projective space.
Let $\mathcal{K}^+$ be a connected component of
\begin{equation} \label{zz} \mathcal{K}=\{ [Z]\in P(V(\Bbb C)) : (Z,Z)=0,\, (Z,\bar Z)<0\},
\end{equation}
and let $O^+_V(\mathbb R)$ be the subgroup of elements in $O_V(\mathbb R)$ which preserve the components of $\mathcal K$.

For $Z\in V(\Bbb C)$, write $Z=X+iY$ with $X,Y\in V(\Bbb R)$.
Given an even lattice $L\subset V$, let $\Gamma \subseteq O^+_L:=O_L\cap O_V^+(\Bbb R) $ be a subgroup of finite index. Then $\Gamma$ acts on $\mathcal{K}$ discontinuously.
Let
$$\widetilde{\mathcal{K}}^+=\{ Z\in V(\Bbb C)-\{0\} : [Z]\in \mathcal{K}^+\}.
$$

Let $k\in  \frac 1 2 \mathbb Z$, and $\chi$ be a multiplier system of $\Gamma$. Then a meromorphic function $\Phi: \tilde{\mathcal{K}}^+\longrightarrow \Bbb C$ is called a {\em meromorphic modular form} of weight $k$ and
multiplier system $\chi$ for the group $\Gamma$, if
\begin{enumerate}
\item $\Phi$ is homogeneous of degree $-k$, i.e., $\Phi(cZ)=c^{-k}\Phi(Z)$ for all $c\in\Bbb C-\{0\}$,

\item $\Phi$ is invariant under $\Gamma$, i.e., $\Phi(\gamma Z)=\chi(\gamma)\Phi(Z)$ for all $\gamma\in \Gamma$.
\end{enumerate}
This definition agrees with the one given in \cite{GN-02}. Since $SO(3,2)$ is isogeneous to $Sp_4$, the automorphic forms on $O(3,2)$ are Siegel modular forms. Similarly, $SO(2,2)$ is isogeneous to $SL_2\times SL_2$, and so the automorphic forms on $O(2,2)$ are Hilbert modular forms.

\subsection{Automorphic correction}\label{automorphic}
A Kac-Moody algebra $\mathfrak g$ is called {\em Lorentzian} if its generalized Cartan matrix 
is given by a set of simple roots of a Lorentzian lattice $M$, namely, a lattice with a non-degenerate integral symmetric bilinear form 
$(\cdot, \cdot)$ of signature $(n,1)$ for some integer $n \ge 1$. Assume that $\mathfrak g$ is a Lorentzian Kac-Moody algebra.
Let $\Pi$ be a set of (real) simple roots of $\mathfrak g$. Then the generalized Cartan matrix $A$ of $\mathfrak g$ is given by
    \[ A= \begin{pmatrix} \frac {2(\alpha, \alpha')}{(\alpha, \alpha)} \end{pmatrix}_{\alpha, \alpha' \in \Pi}.
\]
The Weyl group $W$ is a subgroup of $O(M)$. 
Consider the cone \[ V(M)= \{ \beta \in M \otimes \mathbb R \, | \, (\beta,\beta) <0 \} ,\] which is a union of two half cones. One of these half cones is denoted by $V^+(M)$. The reflection hyperplanes of $W$ partition $V^+(M)$ into fundamental domains, and we choose one fundamental domain $\mathfrak D \subset V^+(M)$ so that the set $\Pi$ of (real) simple roots is orthogonal to the fundamental domain $\mathfrak D$ and 
\[ \mathfrak D = \{ \beta \in \overline{V^+(M)} \, | \, (\beta, \alpha) \le 0  \text{ for all } \alpha \in \Pi \}.
\]
Note that this is a negative Weyl chamber. We have a Weyl vector $\rho \in M \otimes \mathbb Q$ satisfying $(\rho , \alpha) = - (\alpha, \alpha)/2$ for each $\alpha \in \Pi$.

Define the complexified cone $\Omega(V^+(M))= M \otimes \mathbb R + i V^+(M)$. Let $L=\begin{pmatrix} 0 & -m \\ -m & 0 \end{pmatrix} \oplus M$ be an extended lattice for some $m \in \mathbb N$. We consider the quadratic space $V=L \otimes \mathbb Q$ and obtain $\mathcal K^+$ as in \eqref{zz}. Define a map $\Omega(V^+(M)) \rightarrow \mathcal K$ by $z \mapsto \left [ \frac {(z, z)} {2m} \, e_1 +  e_2 + z \right ]$, where $\{ e_1, e_2 \}$ is the basis for $\begin{pmatrix} 0 & -m \\ -m & 0 \end{pmatrix}$. Then the space $\mathcal K^+$ is canonically identified with $\Omega(V^+(M))$.

The denominator of $\mathfrak g$ is $\sum_{w \in W} \det(w) e(-(w (\rho), z))$, which is not an automorphic form on $\Omega(V^+(M))$
in general.
Gritsenko and Nikulin \cite{GN-02, GN-96, GN-97} introduced the concept of automorphic correction, originated in Borcherds' construction of the fake Monster Lie algebra. The idea is to add imaginary simple roots to extend the given Kac-Moody algebra so that the denominator of the extended algebra becomes an automorphic form. Their construction is given by a meromorphic automorphic form
$\Phi(z)$ on $\Omega(V^+(M))$ with respect to a subgroup $\Gamma \subset O^+_L$ of finite index. Following their definition, an automorphic form $\Phi(z)$ is called an {\em automorphic correction} of the Lorentzian Kac-Moody algebra $\mathfrak g$ if it has a Fourier expansion of the form
     \[ \Phi(z) = \sum_{w \in W} \det (w) \left ( e \left ( -  (w (\rho), z) \right )- \sum_{a \in M \cap
     \mathfrak D,\, a\ne 0} m(a) \, e(- (w(\rho+a), z)) \right ), \] where  $e(x)=e^{2 \pi i x}$ and $m(a) \in \mathbb Z$ for
     all $a \in M \cap \mathfrak D$.

 An automorphic correction $\Phi(z)$ defines a generalized Kac-Moody superalgebra $\mathcal G$ as in \cite{GN-02} so that the denominator of $\mathcal G$ is $\Phi(z)$.
In particular, the function $\Phi(z)$ determines the set of imaginary simple roots of  $\mathcal G$ in the following way:
First, assume that  $a \in M\cap \mathfrak D$ and $(a,a)<0$. If $m(a)>0$ then $a$ is an even imaginary simple root with multiplicity $m(a)$, and if $m(a)<0$ then $a$ is an odd imaginary simple root with multiplicity $-m(a)$. Next, assume that $a_0\in M\cap \mathfrak D$ is primitive and $(a_0,a_0)=0$. Then we define $\mu(na_0) \in \mathbb Z$, $n \in \mathbb N$ by 
\[  1 - \sum_{k=1}^\infty m(ka_0) t^k = \prod_{n=1}^\infty (1-t^n)^{\mu(na_0)} , \] where $t$ is a formal variable.
If $\mu(na_0) >0$ then $na_0$ is an even imaginary simple root with multiplicity $\mu(na_0)$; if $\mu(na_0)<0$ then $na_0$ is an odd imaginary simple root with multiplicity $-\mu(na_0)$. 
 
The generalized Kac-Moody superalgebra $\mathcal G$ will be also called an {\em automorphic correction} of $\mathfrak g$.
Using the Weyl-Kac-Borcherds denominator identity for $\mathcal G$, the automorphic form $\Phi(z)$ can be written as the infinite product
 \[ \Phi(z)= e(-(\rho, z)) \prod_{\alpha \in \Delta(\mathcal G)^+} (1 - e(-(\alpha, z)))^{\mathrm{mult}(\mathcal G, \alpha)} ,\] where $\Delta(\mathcal G)^+$ is the set of positive roots of $\mathcal G$ and $\mathrm{mult}(\mathcal G, \alpha)$ is the root multiplicity of 
$\alpha$ in $\mathcal G$.

Here are properties of $\mathcal G$:
\begin{enumerate}
  \item In general, the root multiplicities may be negative. So $\mathcal G$ is a superalgebra. If $\Phi(z)$ is holomorphic, then root multiplicities are positive, and hence $\mathcal G$ is a generalized Kac-Moody algebra.
  \item $\mathcal G$ and $\mathfrak g$ have the same root lattice $M$, the same Weyl vector $\rho$, and the same Weyl group $W$. 
  \item A positive root of $\mathfrak g$ is also a positive root of $\mathcal G$, and in such a case,
  $\mathrm{mult}(\frak g, \alpha)\leq \mathrm{mult}(\mathcal G, \alpha)$. However, not all roots of $\mathcal G$ are roots of $\frak g$.
\end{enumerate}  

Hence finding an automorphic correction of $\mathfrak g$ is, given a sum $\sum_{w \in W} \det (w) e \left ( -  (w (\rho), z) \right )$,
to find $m(a)$ for each $a\in M\cap \mathfrak D$, $a \neq 0$, so that the resulting sum 
$$\sum_{w \in W} \det (w) \left ( e \left ( -  (w (\rho), z) \right )- \sum_{a \in M \cap
     \mathfrak D,\, a\ne 0} m(a) \, e(- (w(\rho+a), z)) \right ),
$$
is an automorphic form on $O(n+1,2)$. It is highly non-trivial. We will demonstrate how non-trivial it is, using a Siegel cusp form and a Hilbert modular form. 

\subsection{Questions} Here are some questions about automorphic correction.
\begin{enumerate}
\item The first natural question is whether an automorphic correction is unique. If it is, then the automorphic form $\Phi(z)$ is determined by the Weyl vector $\rho$ (along with $M$ and $W$), and by the property that its Fourier coefficients are integers. 

\item Given a Lorentzian Kac-Moody algebra $\mathfrak g$, the necessary and sufficient condition for the existence of an automorphic correction is that we have an automorphic form $\Phi(z)$ which has a Fourier expansion of the form 
$\Phi(z)=\sum_{a\in M} m(a) e(-(\rho+a, z))$ with $m(a)\in\Bbb Z$ and $m(0)=1$, and satisfies $\Phi(wz)=\det(w)\Phi(z)$ for $w \in W$. (See the arguments in Section \ref{sec-E10}.) The Weyl vector $\rho$ is minimal in the sense that if the Fourier coefficient $m(a)$ of $e(-(\rho+a, z))$ is non-zero, then $a \in \overline{V^+(M)}$.
 
There are many examples of automorphic forms on $O(n+1,2)$ with integer Fourier coefficients. This gives a heuristic reason for the expectation that every Lorentzian Kac-Moody algebra would have an automorphic correction. 
Now the question is: How can we construct $\Phi(z)$ for a given $\mathfrak g$? We need to know how to determine the level, namely, determine $\Gamma\subset O_L^+$, and the weight, and multiplier system $\chi$ of $\Phi(z)$. Here one may have to consider half-integral weight forms, and the multiplier system can be quite complicated. Any such automorphic form will have infinite product expansion. It is a striking application of Kac-Moody algebras.
\item A related question is whether a given automorphic form with integer Fourier coefficients could be an automorphic correction for a certain Kac-Moody algebra. Bruinier \cite{Bru} proved that a large class of meromorphic forms for $n \ge 2$ can be written as infinite products, called {\em Borcherds products}. It will be interesting to study when these products become automorphic correction.

\item Another question is: How can we determine when $\mathcal G$ is a superalgebra from the original algebra $\frak g$?
Namely, when do we look for a meromorphic automorphic form as automorphic correction for $\frak g$?

\end{enumerate}

\subsection{Rank $3$ hyperbolic Kac-Moody Algebra $\mathcal F$} \label{rank3}

Let $\mathcal F =\mathfrak g(A)$ be the hyperbolic Kac-Moody algebra associated to the generalized Cartan matrix $A=(a_{ij})=\begin{pmatrix} 2 & -2 & 0 \\ -2 & 2 & -1 \\ 0 & -1 & 2 \end{pmatrix}$.
We denote by $S_2(\mathbb C)$ (resp. $S_2(\mathbb Z)$) the set of all symmetric $2 \times 2$ complex (resp. integer) matrices, and define a quadratic form on $S_2(\mathbb Z)$ by $(X,X)=-2\det X$ for $X\in S_2(\mathbb Z)$.
Let $\{ \alpha_1, \alpha_2, \alpha_3 \}$ be the set of simple roots, identified with elements in $S_2(\mathbb Z)$:
$$\alpha_1=\begin{pmatrix} 0&-1\\-1&0\end{pmatrix},\quad \alpha_2=\begin{pmatrix} 1&1\\1&0\end{pmatrix},\quad \alpha_3=\begin{pmatrix} -1&0\\0&1\end{pmatrix}.
$$
Then the imaginary positive roots are identified with positive semi-definite matrices in $S_2(\mathbb Z)$. If $N \in S_2(\mathbb Z)$ is positive semi-definite, we will write $N \ge 0$.  The set of positive real roots is given by \[ \Delta^+_{\mathrm{re}} = \left \{ \begin{pmatrix} n_1 & n_2 \\ n_2 & n_3 \end{pmatrix} \in S_2(\mathbb Z) \Big | n_1 n_3 - n_2^2 = -1, 
n_2 \le n_1 + n_3 , 0 \le n_1 + n_3 , 0 \le n_3
  \right \}  . \]

The Weyl group $W$ of $\mathcal F$ is isomorphic to $PGL_2(\mathbb Z)$ through the map given by
\[ \sigma_1 \mapsto \begin{pmatrix} 1 & 0 \\ 0 & -1 \end{pmatrix}, \quad \sigma_2 \mapsto \begin{pmatrix} -1 & 1 \\ 0 & 1 \end{pmatrix}, \quad \sigma_3 \mapsto \begin{pmatrix} 0 & 1 \\ 1 & 0 \end{pmatrix} ,\] where $\sigma_i$ ($i=1,2,3$) are the simple reflections corresponding to $\alpha_i$.
The Siegel upper half-plane $\mathbb H_2$ of genus $2$ is defined by \[ \mathbb H_2 = \left \{ Z=X+ iY \in S_2(\mathbb C) \ | \ Y\text{ is positive definite} \right \} .\] We will use the coordinates $z_1, z_2, z_3$ for $\mathbb H_2$ so that $Z= \begin{pmatrix} z_1 & z_2 \\ z_2 & z_3 \end{pmatrix} \in \mathbb H_2$, and define the pairing $(X,Z)$ for $X\in S_2(\mathbb Z)$ by $(X,Z)=-\mathrm{tr}(XZ)$.

The denominator identity for $\mathcal F$ is
\begin{eqnarray} & & e(\mathrm{tr}(PZ)) \prod_{0\leq N\in S_2(\mathbb Z)} (1-e( \mathrm{tr}(NZ)))^{\mathrm{mult}(N)} \prod_{N\in \Delta^+_{\mathrm{re}} } (1-e( \mathrm{tr}(NZ))) \nonumber \\ 
& =&  \sum_{g\in PGL_2(\mathbb Z)} \det(g) e( \mathrm{tr}(gPg^t Z)), \label{eqn-denom}
\end{eqnarray}
where $P=\begin{pmatrix} 3&\frac 12\\ \frac 12&2\end{pmatrix}$. The matrix $P$ corresponds to $\rho$, the Weyl vector.  

In \cite[Theorem 1.5]{GN-96},  Gritsenko and Nikulin proved that the Siegel modular form $\Delta_{35}(Z)$ of weight $35$ is an automorphic correction of $\mathcal F$. In order to define $\Delta_{35}(Z)$ as a product, we need to consider some Jacobi forms first.
For $k \ge 4$ even, we define the {\em Jacobi-Eisenstein series of weight $k$ and index $m$} by
\[ E_{k,m}(\tau ,z) = \tfrac 1 2  \sum_{c,d \in \mathbb Z \atop (c,d)=1} \sum_{\lambda \in \mathbb Z} (c \tau + d)^{-k} \ e \left ( m\lambda^2 \frac {a \tau +b}{c \tau +d} + 2m \lambda \frac z {c \tau +d} - \frac {cm z^2}{c \tau +d} \right ) ,\] where $a, b$ are chosen so that $\begin{pmatrix} a& b \\ c& d\end{pmatrix} \in SL_2(\mathbb Z)$. We also consider a Jacobi form of weight $12$ and index $1$:
\[ \phi_{12, 1}(\tau, z) = \tfrac 1 {144} \left ( E^2_4(\tau) E_{4,1}(\tau, z) -E_6(\tau) E_{6, 1}(\tau, z) \right ) , \]
where $E_k(\tau)$ are the usual Eisenstein series of weight $k$ defined by \[ E_{k}(\tau) = \tfrac 1 2  \sum_{c,d \in \mathbb Z \atop (c,d)=1} (c \tau + d)^{-k} .\]

Now we define a weak Jacobi form $\phi_{0,1}(\tau , z )$ of weight $0$ and index $1$ by
\begin{equation} \label{phi}
\phi_{0,1}(\tau , z ) = \frac { \phi_{12, 1}(\tau, z )}{\Delta_{12}(\tau)} =  \sum_{n=0}^\infty \sum_{r \in \mathbb Z} c(n,r) \ e( n \tau + r z ),
\end{equation}
where $\Delta_{12}(\tau) = e(\tau) \prod_{n \ge 1} (1 - e(n\tau))^{24}$ and $c(n,r)$ are the Fourier coefficients.
Since $c(n,r)$ depends only on $4n -r^2$, the following function is well-defined:
\begin{equation*}
c(N) = \begin{cases} c(n,r) & \text{ if } N= 4n -r^2, \\ 0 & \text{ otherwise.} \end{cases}
\end{equation*}
In particular, we have $c(0)=10, c(-1)=1$  and $c(n)=0$ for $n <-1$. We use the function $c(N)$ to define
\begin{equation} \label{eqn-f2} c_2(N) = 8c(4N) + 2 \left ( \left (\tfrac{-N}{2} \right ) -1 \right ) c(N) +c \left ( \tfrac N 4 \right ) ,\end{equation}
where we put
\[ \left ( \frac D 2 \right ) = \left \{ \begin{array}{rl} 1 & \text{for } D \equiv 1 \ (\text{mod }8), \\ - 1 & \text{for } D \equiv 5 \ (\text{mod }8), \\ 0 & \text{for } D \equiv 0 \ (\text{mod }2) . \end{array} \right . \]

For $Z = \begin{pmatrix} z_1 & z_2 \\ z_2 & z_3 \end{pmatrix} \in \mathbb H_2$, we set $q=e(z_1)$, $r=e(z_2)$ and $s=e(z_3)$. Let $S_2(\frac 1 2 \mathbb Z)$ be the set of symmetric half integral $2 \times 2$ matrices. Then we have 

\begin{equation} \label{eqn-Delta} \Delta_{35}(Z) = q^3 r s^2 \prod_{(n,l,m) \in \mathcal D} (1-q^nr^ls^m)^{c_2(4nm-l^2)}=\sum_{0<T\in S_2(\frac 1 2\mathbb Z)} A(T) e( \mathrm{tr}(TZ)),
\end{equation}
where $T>0$ means positive definite, and the integers $c_2(N)$ are defined in (\ref{eqn-f2}), and we denote by $\mathcal D$ the set of integer triples $(n,l,m) \in \mathbb Z^3$ such that (1) $(n,l,m) =(-1, 0, 1)$ or (2) $n \ge 0, m \ge 0$ and either $n+m>0$ and $l$ is arbitrary or $n=m=0$ and $l<0$. Since $\Delta_{35}(Z)$ is holomorphic, the root multiplicities $c_2(N)$ are positive, and hence the automorphic correction $\mathcal G$ is a generalized Kac-Moody algebra.

The group $PGL_2(\mathbb Z)$ acts on $S_2(\mathbb R)$ by $g(S)=gS g^t$. Then
a fundamental domain $\mathfrak D$ is given by the negative Weyl chamber of $\mathcal F$ through our identification of simple roots with matrices in $S_2(\mathbb R)$. Explicitly, we obtain
$$\mathfrak D=\left \{ \begin{pmatrix} y_1&y_2\\y_2&y_3\end{pmatrix}\in S_2(\mathbb R) \Big |\, \text{$0\leq 2y_2\leq y_3\leq y_1$; if $y_2=0$, then $0<y_3\leq y_1$} \right \}.
$$
Note that $\mathfrak D$ is the fundamental domain defined in Section \ref{automorphic}. Since $A(gTg^t)=\det(g) A(T)$ for $g\in PGL_2(\mathbb Z)$, we can write $\Delta_{35}(Z)$ as
$$\Delta_{35}(Z)=\sum_{T\in S_2(\frac 12\mathbb Z)\cap \mathfrak D} A(T) \sum_{g\in PGL_2(\mathbb Z)} 
\det(g) e(\mathrm{tr}(gTg^t Z)).
$$

Note that the inner sum for $T=P$ is the summation side of the  denominator identity \eqref{eqn-denom} of the hyperbolic Kac-Moody algebra $\mathcal F$.
Hence the automorphic correction is to find $A(T)$ for all $T\in S_2(\mathbb Z)\cap \mathcal D$ so that
the resulting sum
$$\sum_{T\in \mathfrak D} A(T) \sum_{g\in PGL_2(\mathbb Z)} \det(g) e(\mathrm{tr}(gTg^t Z))
$$
is modular. It is highly non-trivial from the point of view of automorphic forms.

\subsection{Rank 2 hyperbolic Kac-Moody algebras}

Let $A=\begin{pmatrix} 2 & -3 \\ -3 & 2 \end{pmatrix}$ be a generalized Cartan matrix, and  $\mathcal H(3)$ be the hyperbolic Kac-Moody algebra associated with the matrix $A$.
Let $F=\mathbb Q[\sqrt{5}]$ and $\mathcal O$ be the ring of integers of $F$. We choose a fundamental unit $\varepsilon_0=\frac {1+\sqrt{5}}2$ and set $\eta = \varepsilon_0^2 = \frac {3+\sqrt{5}}2$. 

The roots of $\mathcal H(3)$ can be identified with elements of the inverse different $\mathfrak{d}^{-1} = \frac 1 {\sqrt 5} \mathcal O$ as follows. The set of positive real roots is given by 
\[ \Delta^+_{\mathrm{re}} = \left \{ \frac 1 {\sqrt{5}}  \eta^{j} \ (j > 0),  \qquad
 -\frac 1 {\sqrt{5}}  \bar \eta^{j} \ (j \ge 0)  \right \} , \]  where $\bar x$ is the conjugate of $x$ in $F$. The set of positive imaginary roots is given by
\begin{equation} \label{imroot} \Delta^+_{\mathrm{im}} = \left \{  \frac 1 {\sqrt{p}} \, \eta^{j} (m \eta -n) ,  \
\frac 1 {\sqrt{5}}  \, \eta^{j}(n \eta -m) , \
\frac 1 {\sqrt{5}} \, \bar \eta^{j}(n - m \bar \eta ) , \
\frac 1 {\sqrt{5}} \,  \bar \eta^{j}(m  - n \bar \eta) \right \} , \end{equation}
where $j \ge 0$ and $(m,n) \in \Omega_k$ for $k \ge 1$ and the set $\Omega_k$ is defined to be 
\[  \left \{ (m,n) \in \mathbb Z_{\ge 0} \times \mathbb Z_{\ge 0} : \sqrt{\frac {4 k}{5} } \le m \le \sqrt { k },\ n = \frac { 3m-\sqrt{5m^2 -4k}} 2 \right \}. \] 
See  \cite{LM, Fein, KaMe, KimLee-1} for more details. The Weyl group $W$ also acts on $F$; in particular, the simple reflections $r_1$ and $r_2$ are given by
\[ r_1 x = \eta^2 \bar x \quad \text{ and } \quad r_2x= \bar x \qquad \text{for } x \in F .\] Then the Weyl group is identified with the semidirect product of multiplication by $\eta^{2n}$, $n \in \mathbb Z$ and conjugation, i.e. $W \cong \{ \eta^{2n} : n \in \mathbb Z \} \rtimes \{ \bar \cdot \}$. 

Let $\mathbb H$ be the upper half plane. We define a paring on $F \times \mathbb H^2$ by
\[ (\nu, z)= -5\,(\nu z_2+ \bar \nu z_1) \qquad \text{for } \nu \in F \text{ and } z=(z_1, z_2) \in \mathbb H^2, \] and consider the denominator function of $\mathcal H(3)$ as a function on $\mathbb H^2$. Then  the denominator identity is, for $z \in \mathbb H^2$,
$$
e(-(\rho, z)) \prod_{\alpha\in \Delta_+} (1-e(-(\alpha,z)))^{\mathrm{mult}(\alpha)}
=\sum_{w\in W} \det(w) e(-(w\rho, z)) ,
$$
where we have
\[  \rho = \frac 1 5 (1+\eta)= \frac {\varepsilon_0}{\mathrm{tr}(\varepsilon_0)\sqrt 5} = \frac {1+\sqrt 5}{2 \sqrt 5} .\]

Consider the Hilbert modular form $\Psi$ of weight $5$ defined by
\[\Psi(z)= e\left( \frac {\varepsilon_0 z_1}{\sqrt{5}}-\frac {\bar \varepsilon_0 z_2}{\sqrt{5}}\right)\prod_{\substack{\nu\in\frak{d}^{-1} \\ \nu \gg 0} }
(1-e(\nu z_1+\bar \nu z_2))^{s(p\nu\bar \nu)a(p\nu\bar \nu)} \prod_{\substack{\nu\in\frak{d}^{-1} ,\ \nu+ 2\bar\nu>0 \\ N(\nu)=-m^2/5} }
(1-e(\nu z_1+\bar \nu z_2)) \]
for $z=(z_1, z_2 ) \in \mathbb H^2$, where $s(n)=1$ if $5\nmid n$ or $s(n)=2$ otherwise, and $a(n)$ is the Fourier coefficient of the weakly holomorphic modular form $f$ of weight $0$ for the group $\Gamma_0(5)$ with principal part $q^{-1}$:
$$f(\tau)=q^{-1}+\sum_{n=0}^\infty a(n)q^n= q^{-1}+5+11q-54q^4+55q^5+44q^6-395q^9+340q^{10}+\cdots.
$$
The function $\Psi(z)$ has a Fourier expansion of the form
\[ \Psi(z) = \sum_{\nu\in\frak{d}^{-1} }  A(\nu) e(\nu z_1+\bar \nu z_2).\]

It is known (cf. \cite{BB}) that $\Psi(z)$ is a meromorphic cusp form and skew-symmetric, i.e., $\Psi(z_2,z_1)=-\Psi(z_1,z_2)$.
We set $\Phi(z)=\Psi(5z_2, 5z_1)=-\Psi(5z_1,5z_2)$. In \cite{KimLee-1}, it is shown that $\Phi(z)$ is an automorphic correction of 
$\mathcal H(3)$, and we can write
\[\Phi(z)= e(-(\rho, z))\prod_{\substack{\nu\in\frak{d}^{-1} \\ \nu \gg 0} }
(1-e(-(\nu,z)))^{s(p\nu\bar \nu)a(p\nu\bar \nu)} \prod_{\alpha\in \Delta^+_{\mathrm{re}}}
(1-e(-(\nu,z))) =  \sum_{\nu\in\frak{d}^{-1}}  A(\nu) e( -(\nu, z)) .\]
Since $\Phi(z)$ is meromorphic, the automorphic correction $\mathcal G$ is a superalgebra.

Let $\mathfrak D$ be a fundamental domain of $\mathfrak{d}^{-1}$ by the action of $W$. 
Since $\Phi(wz)=\det(w) \Phi (z)$ for $w \in W$, we have $A(w\nu)=\det(w)A(\nu)$, and we can write $\Phi(z)$ as
$$\sum_{\nu\in \mathfrak D} A(\nu) \sum_{w\in W} \det(w) e( -(w\nu, z)).
$$
Note that the inner sum when $\nu=\rho$ is the denominator of the hyperbolic Kac-Moody algebra $\mathcal H(3)$.
So finding an automorphic correction of $\mathcal H(3)$ is highly non-trivial.


\section{Kac-Moody Algebra $E_{10}$} \label{hyperbolic}

In this section we fix our notations for the hyperbolic Kac-Moody algebra $E_{10}$. We will follow the notational conventions in \cite{FKN} for the root system of $E_{10}$.

\medskip

Let $\mathbb O$ be the normed division algebra of octonians, consisting of the elements of the form
\[  x_0 + \sum_{i=1}^7 x_i e_i, \qquad  x_i \in \mathbb R,  \]
and equipped with the multiplication satisfying the relations
\[ e_i^2 = -1, \quad e_i e_j = - e_j e_i \ ( i\neq j), \quad e_i e_{i+1} e_{i+3} = -1, \] where the indices are to be taken modulo seven.
Let $H_2(\mathbb O)$ be the Jordan algebra of all Hermitian $2 \times 2$ matrices over $\mathbb O$, i.e.
\[ H_2(\mathbb O) = \left \{ X = \begin{pmatrix} x^+ & z \\ \bar z & x^- \end{pmatrix} : x^+, x^- \in \mathbb R, \ z \in \mathbb O \right \}. \]
We define a quadratic form on $H_2(\mathbb O)$ by
\[ ||X||^2= -2 \det (X) = -2( x^+x^- - z\bar z), \] and obtain the corresponding symmetric bilinear form $(X, Y) = \frac 1 2 ( || X+Y||^2 - ||X||^2 -||Y||^2)$, $X, Y \in H_2(\mathbb O)$. That is, if $X= \begin{pmatrix} x^+ & z \\ \bar z & x^- \end{pmatrix}$ and $Y=\begin{pmatrix} y^+ & w \\ \bar w & y^- \end{pmatrix}$,  then
\[ (X,Y)= -x^+y^- - y^+x^- + z \bar w + w\bar z.\]

We choose the following octonionic units to be the simple roots of the lattice $E_8$:
\[ \begin{array} {lll} a_1 = e_3, & \phantom{LLL}& a_2 = \frac 1 2 ( -e_1 -e_2 -e_3 + e_4) , \\ a_3 = e_1, & & a_4 = \frac 1 2 ( -1 -e_1 -e_4 + e_5) , \\a_5 = 1, & & a_6 = \frac 1 2 ( -1 -e_5 -e_6 - e_7) , \\a_7 = e_6, & & a_8 = \frac 1 2 ( -1 +e_2 +e_4 + e_7) . \end{array} \]
The root lattice spanned by these simple roots gives all octonionic integers called {\em octavians}, and we will denote the lattice by $\mathsf O$. Thus we make the identification $\mathsf O \cong E_8$. The highest root is given by
\[ \theta = 2 a_1 + 3 a_2 + 4 a_3 + 5a_4 +6 a_5 + 4 a_6 +2 a_7 + 3a_8 = \tfrac 1 2 (e_3 + e_4 + e_5 - e_7) .\] 

We consider a lattice $\Lambda $ in $H_2( \mathbb O)$  given by
\[ \Lambda = \left \{  X = \begin{pmatrix} x^+ & z \\ \bar z & x^- \end{pmatrix} : x^+, x^- \in \mathbb Z, \ z \in \mathsf O \right \},  \] and choose vectors
\[ \alpha_{-1} = \begin{pmatrix} 1 & 0 \\ 0 & -1 \end{pmatrix}, \quad \alpha_{0} = \begin{pmatrix} -1 & - \theta  \\ -\bar \theta  & 0 \end{pmatrix}, \quad \alpha_{i} = \begin{pmatrix} 0 & a_i \\ \bar a_i & 0 \end{pmatrix}, 1 \le i \le 8  . \] Then $\{\alpha_{-1}, \alpha_0, \alpha_1, \dots , \alpha_8 \}$ is a basis for $\Lambda$ and the matrix $\left ( \frac {2 (\alpha_i, \alpha_j)}{(\alpha_i, \alpha_i)} \right )$ is the generalized Cartan matrix of $E_{10}$. We will denote by $\mathfrak g$ the corresponding Kac-Moody algebra of $E_{10}$. We define $\Lambda^+$ to be the additive monoid generated by $\alpha_{-1}, \alpha_0, \alpha_1, \dots , \alpha_8$, i.e. \[\Lambda^+ := \mathbb Z_{\ge 0} \, \alpha_{-1} + \mathbb Z_{\ge 0} \, \alpha_0 + \mathbb Z_{\ge 0} \, \alpha_1 + \cdots + \mathbb Z_{\ge 0} \, \alpha_8 . \]
We write $z  \ge w$ for $z,w \in \Lambda $ if $z-w \in \Lambda^+$. 

The Weyl group $W$ of $\mathfrak g$ is generated by the simple reflections
\[  w_i (X) = X - \frac {2 (X, \alpha_i)}{(\alpha_i, \alpha_i)} \, \alpha_i, \qquad i = -1, 0, 1, \dots , 8 , \ X \in H_2(\mathbb O).\] 

Now we determine the set $\Delta^+_{\mathrm{re}}$ of positive real roots and the set $\Delta^+_{\mathrm{im}}$ of positive imaginary roots of $\mathfrak g$. We write $\Delta^+= \Delta^+_{\mathrm{re}} \cup \Delta^-_{\mathrm{im}}$. Consider $X= \begin{pmatrix} x^+ & z \\ \bar z & x^- \end{pmatrix} \in \Lambda$. From Proposition 5.10 in \cite{Kac}, we have that $X$ is a real root $\Leftrightarrow \det X = -1$ and that $X$ is an imaginary root $\Leftrightarrow \det X \ge 0$.  We have \[ X \in \Delta^+ \quad  \Longleftrightarrow \quad \det X \ge -1 \text{ and } X \in \Lambda^+,\] and it is easy to see that $X \in \Lambda^+$ if and only if $x^- \le0$, $x^+ +x^- \le 0$, and $z\ge (x^+ +x^-) \theta$.

The Weyl vector $\rho$ is given by
\[ \rho=30\alpha_{-1}+61\alpha_{0}+93\alpha_{1}+126\alpha_{2}+160\alpha_{3}
+195\alpha_{4}+231\alpha_{5}+153\alpha_{6}+76\alpha_{7}+115\alpha_{8}. \] 
Put in the matrix notation, it becomes
\begin{equation} \label{eqn-rho}
\rho = \begin{pmatrix} -31 & -\rho_K \\ -\overline{\rho_K} &  -30\end{pmatrix} \in \Lambda, 
\end{equation} where $\rho_K=\frac 1 2(1+e_1+11e_2+e_3+15e_4+19e_5+e_6-23e_7) \in \mathsf O$. Under the identification 
$\mathsf O \cong  E_8$,  $\rho_K$ corresponds to $\rho_{E_8}$.
We have $(\rho, \alpha_i)=-1$ for each $i=-1, 0, 1, \dots, 8$ and $(\rho, \rho)=-1240$, and $(\rho_K, \alpha_i)=1$ for each $i= 1, \dots, 8$. We obtain another criterion for positive roots:
\begin{equation} \label{eqn-pos} X \in \Delta^+ \quad  \Longleftrightarrow \quad \det X \ge -1 \text{ and } (\rho,X) <0.\end{equation}

It is proved in \cite{FKN} that the simple reflections $w_{-1}$, $w_0$,...,$w_i$, are given by
$$w_i(X)=M_i\bar X \bar M_i^t, \quad i=-1,0,...,8,
$$
where
$$M_{-1}=\begin{pmatrix} 0&1\\1&0\end{pmatrix},\quad M_0=\begin{pmatrix} -\theta&1\\0&\bar\theta\end{pmatrix},\quad M_i=\begin{pmatrix} a_i&0\\0& -\bar a_i\end{pmatrix}
$$ and  $a_i\in \mathsf O$, $i=1,...,8$, corresponds to the simple roots of $E_8$ as before.

Now $W=W^+\rtimes \{w_{-1}\}$, where $W^+$ is the even part, i.e. the subgroup of elements of even length. Here $W^+$ is generated by
$s_0=w_{-1}w_0$, $s_i=w_{-1}w_i$, $i=1,...,8$. Then
$$s_i(X)=S_i X \bar S_i^t,\quad S_0=\begin{pmatrix} 0&\theta\\-\bar\theta&1\end{pmatrix},\quad
S_i=\begin{pmatrix} 0&-a_i\\ \bar a_i&0\end{pmatrix}.
$$
For $s\in W^+$, write $s=s_{i_1}\cdots s_{i_k}$, and define
$$s(X)=s_{i_1}\cdots s_{i_k}(X).
$$
Then formally we can write $W^+=PSL(\mathsf O)$.

We defined the simple roots in such a way that a fundamental domain of $H_2(\mathbb O)$ under the Weyl group action is given by

\begin{equation} \label{eqn-lo} \mathfrak D = \{ X \in H_2(\mathbb O) : (X, \alpha_i) \le 0 \text{ for } i =-1,0, 1, \dots , 8 \}. \end{equation}
It is a negative Weyl chamber of $E_{10}$, namely,
$$\mathfrak D=\{ r_{-1}\lambda_{-1}+r_0\lambda_0+\sum_{i=1}^8 r_i\lambda_i :\, r_i\leq 0,\ i =-1,0, 1, \dots , 8 \},
$$
where $\lambda_i$'s are the fundamental weights, i.e., $(\lambda_i,\alpha_j)=\delta_{ij}$. Since $\mathsf O$ is self-dual,
we have
$$\lambda_{-1}=\begin{pmatrix} 1&0\\0&0\end{pmatrix},\quad \lambda_0=\begin{pmatrix} 1&0\\0&1\end{pmatrix},\quad
\lambda_i=\begin{pmatrix} 1& \frac 12 a_i\\ \frac 12 \bar a_i&1\end{pmatrix}.
$$
Hence we have
$$\mathfrak D=\left\{ \begin{pmatrix} x^+ & z \\ \bar z & x^- \end{pmatrix} :\, x^+\leq x^-\leq \sum_{i=1}^8 r_i\leq 0, \,  
2z=\sum_{i=1}^8 r_i a_i\right\}.
$$


\section{Automorphic correction of $E_{10}$} \label{sec-E10}

Let $M$ be the even unimodular Lorentzian lattice $II_{s+1,1}$ and we let $L=M\oplus II_{1,1}$. Borcherds proved the following theorem.

\begin{Thm} \cite[Theorem 10.1]{Bor-95} \label{thm-B}
Suppose that $f(\tau)=\sum_n c(n) q^n$ is a weakly holomorphic modular form of weight $-s/2$ for $SL_2(\mathbb Z)$ with integer coefficients, with $24|c(0)$ if $s=0$. Choose a negative norm vector $v_0 \in M\otimes \Bbb R$. Then there is a unique vector $\rho_{v_0} \in M$ such that
\begin{equation} \label{eqn-Bor}
 \Phi (z)= e(-(\rho_{v_0}, z)) \prod_{(r,v_0)>0} (1-e(-(r, z)))^{c(-(r,r)/2)}\end{equation}
is a meromorphic automorphic form of weight $c(0)/2$ for $O_L^+$. 
\end{Thm}

We consider the case when $s=8$ and the weakly holomorphic modular form is defined by
\[f(\tau) = E_4(\tau)^2/\Delta_{12}(\tau) = q^{-1}+504+73764q+2695040q^2+\cdots .\] 
Borcherds mentions this modular form in Section 16 of \cite{Bor-95}.
We write $M=II_{9,1}=E_8\oplus II_{1,1}$.
 Then
we have an isomorphism $\nu: M \overset\sim\longrightarrow \Lambda$ defined by \[ \nu: (z, x^+, x^-) \mapsto \begin{pmatrix} -x^+ & -z \\ -\bar z & -x^- \end{pmatrix} .\] Furthermore, a set of simple roots of the lattice $M$ and its Dynkin diagram are given by those of the Kac-Moody algebra $E_{10}$ considered in Section \ref{hyperbolic} through the identification $\Lambda \cong M$. (See also \cite{Con}). Thus the Kac-Moody algebra $E_{10}$ is Lorentzian.

We take $v_0=(-\rho_{E_8}, -m, -d/24) \in M$ where  $m$ is the constant term of $E_4(\tau)f(\tau) E_2(\tau)/24$, and $d$ is the constant term of $E_4(\tau)f(\tau)$. One can check $m=30$, $d=744$ and \[ \rho_{E_8}= 29\alpha_1 +57 \alpha_2 +84 \alpha_3+110 \alpha_4+135 \alpha_5+91 \alpha_6+46 \alpha_7+68 \alpha_8 .\]  From Theorem 10.4 in \cite{Bor-95}, the vector $\rho_{v_0}$ in Theorem \ref{thm-B} is given by \[\rho_{v_0}=(\rho_K, m, d/24)=(\rho_{E_8}, 30, 31)=-v_0 .\]  
Therefore, the vector $\rho_{v_0}$ in Theorem \ref{thm-B} exactly corresponds to  the Weyl vector $\rho$ of $E_{10}$ in \eqref{eqn-rho}, i.e. we have
\begin{equation} \label{eqn-90}
 \nu(\rho_{v_0})= \rho = \begin{pmatrix}
-31 & - \rho_K \\ -\bar \rho_K & -30
\end{pmatrix} \in \Lambda \quad \text{and} \quad \nu(v_0)= -\rho .\end{equation}
Consequently, if $\nu(r)=\alpha$, we get
\begin{equation} \label{eqn-12}
 (r,v_0)>0 \ \Longleftrightarrow \ (\alpha , \rho)<0.
 \end{equation} 

We denote by $\Phi_f(z)$ the automorphic form given by the Borcherds product \eqref{eqn-Bor}. With all the preparations we made in the above, we can now prove the following theorem.

\begin{Thm}
The function $\Phi_f(z)$ is an automorphic correction for the Kac-Moody algebra $E_{10}$.
\end{Thm}

\begin{proof}
We identify $M$ with $\Lambda$ via the isometry $\nu$ as before. Write $f(\tau)=\sum_n c(n) q^n$, and note that $-(r,r)/2= -(\nu(r), \nu(r))/2=\det \nu(r)$ for $r \in M$.  Since $c(n)=0$ for $n \le -2$, we obtain from \eqref{eqn-pos} and \eqref{eqn-12}
\begin{equation} \label{eqn-gh} (r, v_0)>0 \ \text{ and } \ c(-(r,r)/2) \neq 0 \quad \Longleftrightarrow \quad \nu(r) \in \Delta^+, \end{equation} where $\Delta^+$ is the set of positive roots of $E_{10}$.
Therefore, we obtain from \eqref{eqn-Bor} and \eqref{eqn-90}
\[
 \Phi_f (z)= e(-(\rho, z)) \prod_{\alpha \in \Delta^+} (1-e(-(\alpha, z)))^{c(-(\alpha,\alpha)/2)}.
\]
 We claim that $\Phi_f(w_i z) = - \Phi_f(z)$ for each $i=-1,0,1, \dots, 8$. Indeed, since $(\rho, \alpha_i)=-1$, we have
\begin{align*} \Phi_f(w_i z)&= e(-(\rho, w_iz)) \prod_{\alpha \in \Delta^+} (1-e(-(\alpha, w_iz)))^{c(-(\alpha , \alpha)/2)}\\ &= e(-(w_i\rho, z)) \prod_{\alpha \in \Delta^+} (1-e(-(w_i \alpha, z)))^{c(-(\alpha , \alpha)/2)}\\ &= e(-(\rho +\alpha_i, z))\frac {1-e((\alpha_i,z))}{1-e(-(\alpha_i,z))} \prod_{\alpha \in \Delta^+} (1-e(-(\alpha, z)))^{c(-(\alpha,\alpha)/2)}\\
&= - \Phi_f(z).
\end{align*}
Then we have $\Phi_f(wz)=\det(w) \Phi_f(z)$ for $w \in W$.

Now we use the same argument as in the proof of Theorem 5.16 in \cite{KimLee-1} and obtain that
\[ \Phi_f(z) = \sum_{w \in W} \det (w) \left ( e \left ( -  (w (\rho), z) \right )- \sum_{a \in \Lambda \cap
     \mathfrak D,\, a\ne 0} m(a) \, e(- (w(\rho+a), z)) \right ), \]
where $\mathfrak D$ is defined in \eqref{eqn-lo} and $m(a) \in \mathbb Z$ for all $a \in \Lambda \cap \mathfrak D$. 
It completes the proof.

\end{proof}

\subsection{Some remarks}

Note that $f(\tau)=\sum_n c(n)q^n$ has positive Fourier coefficients. Hence the automorphic correction $\mathcal G$ of $E_{10}$ is a generalized Kac-Moody algebra with $\Phi_f$ as the denominator.
Moreover, from \eqref{eqn-gh}, we see that
the set of positive roots of $E_{10}$ coincides with the set of positive roots of $\mathcal G$. 
Let $\delta$ be the minimal null root, i.e., $(\delta,\delta)=0$. Then $\text{mult}(\delta)=8$ in $E_{10}$.
However, it follows from $c(0)=504$ that $\text{mult}(\delta)=504$ in $\mathcal G$. The extra $496$ dimensions come from imaginary simple roots.
Since $\Phi_f(wz)=\det(w) \Phi_f(z)$ for $w \in W$,  we can write, using the fundamental domain $\mathfrak D$, the function $\Phi_f(z)$ as
$$\Phi_f(z)=\sum_{r\in \mathfrak D} A(r) \sum_{w\in W} \det(w)e(-(w(r),z)).
$$
The inner sum when $r=\rho$ is the denominator of $E_{10}$.

\medskip

We obtain asymptotics of Fourier coefficients of the modular form $f(z)$, using the method of Hardy-Ramanujan-Rademacher.
Recall that
\[f(\tau) = E_4(\tau)^2/\Delta(\tau) =q^{-1}+\sum_{n=0}^\infty c(n)q^n=q^{-1}+504+73764q+2695040q^2+\cdots .\] 
The function  $f$ is a modular form of weight $-4$ for the full modular group $SL_2(\mathbb Z)$.
Therefore as in \cite{KimLee}, we can show
\[
c(n) = 2\pi  n^{-\frac 5 2}I_5(4\pi\sqrt{n})+O\left(n^{-\frac 9 4}\log(4\pi \sqrt n)I_5\left(2\pi\sqrt n\right)\right),
\]
where $I_5(x)$ is the Bessel $I$-function. It has an asymptotic expansion $I_5(x)=\frac {e^x}{\sqrt{2\pi x}}(1+O(\frac 1x)).$

\medskip

\subsection{Other generalized Kac-Moody algebras containing $E_{10}$}

We close this paper with mentioning two other generalized Kac-Moody algebras which contain $E_{10}$. We can compare these algebras with the automorphic correction $\mathcal G$ of $E_{10}$. These generalized Kac-Moody algebras are not automorphic correction, but they provide some upper bounds for root multiplicities of $E_{10}$. 

\subsubsection*{{\rm (1)} The algebra $\mathfrak g_{II_{9,1}}$ of physical states}

Let $M$ be a nonsingular even lattice, and $\mathfrak V(M)$ be the vertex algebra associated to $M$. Then we have the Virasoro operators $L_i$ on $\mathfrak V(M)$ for each $i \in \mathbb Z$. We define the {\em physical space} $P^n$ for each $n \in \mathbb Z$ to be the space of vectors $v \in \mathfrak V(M)$ such that
\[  L_0(v) = nv \qquad \text{ and } \qquad L_i(v)=0 \ \text{ for } i >0 .\] 
Then the space $\mathfrak g_M:= P^1/L_{-1} P^0$ is a Lie algebra
 and satisfies the following properties \cite{Bor-86}:
\begin{enumerate}
\item[(i)] Let $\mathfrak g$ be a Kac-Moody Lie algebra with a generalized Cartan matrix that is indecomposable, simply laced and non-affine. If the lattice $M$ contains the root lattice of $\mathfrak g$ then $\mathfrak g$ can be mapped into $\mathfrak g_M$ so that the kernel is in the center of $\mathfrak g$.

\item[(ii)] Let $d$ be the dimension of $M$, and $\alpha \in M$ be a root such that $(\alpha , \alpha) \le 0$. Then the root multiplicities of $\alpha$ in $\mathfrak g_M$ is equal to \begin{equation} \label{eqn-bound} p^{(d-1)}(1-(\alpha|\alpha)/2) - p^{(d-1)}(-(\alpha|\alpha)/2),\end{equation}
where $p^{(\ell)}(n)$ is the number of multi-partitions of $n$ into parts of $\ell$ colors. 
\end{enumerate}

Therefore, when we have a hyperbolic Kac-Moody algebra $\mathfrak g$ with root lattice $M$, the Lie algebra $\mathfrak g_M$ contains $\mathfrak g$ and provides an upper bound \eqref{eqn-bound} for root multiplicities of $\mathfrak g$. Moreover, it is well known that the Lie algebra $\mathfrak g_M$ is a generalized Kac-Moody algebra.

In particular, we can apply the above construction to the lattice $M=II_{9,1}$. Since $M$ is the root lattice of $E_{10}$ as we have seen in the previous sections, the algebra $\mathfrak g_{II_{9,1}}$ contains $E_{10}$ and the root multiplicity is given by $p^{(9)}(n)-p^{(9)}(n-1)$.
H. Nicolai and his coauthors studied the algebra $\mathfrak g_{II_{9,1}}$ extensively. For example, see \cite{BGGN, BGN}.

\subsubsection*{{\rm (2)} Niemann's algebra $\mathcal G_2$}

In his Ph.D. thesis \cite{Nie}, P. Niemann constructed a generalized Kac-Moody algebra $\mathcal G_{2}$ which contains $E_{10}$.

Let $\Lambda$ be the Leech lattice, i.e. the $24$-dimensional positive definite even unimodular lattice with no roots, and let $\hat \Lambda$ be the central extension of $\Lambda$ by a group of order $2$. Assume that $\sigma \in \mathrm{Aut}(\hat \Lambda)$ is of order $2$ and cycle shape $1^82^8$. Then we define the lattice $M=\Lambda^\sigma \oplus II_{1,1}$, where $\Lambda^\sigma$ is the fixed point lattice. The Weyl vector $\rho$ is given by $\rho=(0,0,1) \in M$ and we denote by $W^\sigma$ the full reflection group of the lattice $M$. We define $p_\sigma(n)$ by
$q \eta(z)^{-8} \eta(2z)^{-8} = \sum_{n=0}^\infty p_\sigma(n) q^n$, $q=e^{2 \pi i z}$, where $\eta(z)$ is Dedekind's eta-function.

With these data, Niemann \cite{Nie} constructed a generalized Kac-Moody algebra $\mathcal G_2$ and
proved the following twisted denominator identity of $\mathcal G_2$:
\begin{align} & e^\rho \prod_{r \in M^+} (1-e^r)^{p_\sigma(1-(r,r)/2)} \prod_{(2{M^*})^+} (1-e^r)^{p_\sigma(1-(r,r)/4)} \label{eqn-pq} \\ &= \sum_{w \in W^\sigma} \det(w) w \left ( e^\rho \prod_{i=1}^\infty (1-e^{i\rho})^8(1-e^{2i\rho})^8 \right ), \nonumber
\end{align}
where $M^*$ is the dual of $M$. Moreover he showed that $\mathcal G_2$ contains $E_{10}$. Consequently, it follows from \eqref{eqn-pq} that 
\[ \mathrm{mult}(E_{10}, \alpha) \le \begin{cases} p_\sigma ( 1 - \tfrac 1 2 (\alpha , \alpha) ) & \text{ if } \alpha \in M, \alpha \notin 2 M^* , \\ p_\sigma ( 1 - \tfrac 1 2 (\alpha , \alpha) ) +  p_\sigma ( 1 - \tfrac 1 {4} (\alpha , \alpha) ) & \text{ if } \alpha \in 2 M^* .
\end{cases}
\]



\begin{thebibliography}{10}

\bibitem{BGGN} O. B{\"a}rwald, R. W. Gebert, M. G{\"u}naydin and H. Nicolai, {\em Missing modules, the gnome Lie algebra, and $E_{10}$}, Comm. Math. Phys. {\bf 195} (1998), no. 1, 29--65. 

\bibitem{BGN} O. B{\"a}rwald, R. W. Gebert and H. Nicolai, {\em On the imaginary simple roots of the Borcherds algebra $\mathfrak g_{II_{9,1}}$}, Nuclear Phys. {\bf B 510} (1998), no. 3, 721--738.

\bibitem{BHK} E.~A. Bergshoeff, O. Hohm, A. Kleinschmidt, H. Nicolai, T.~A. Nutma and J. Palmkvist, {\em $E_{10}$ and gauged maximal supergravity}, J. High Energy Phys. 2009, no. 1, 020. 

\bibitem{Bor-86} R. E. Borcherds,  {\it Vertex algebras, Kac-Moody algebras, and the Monster}, Proc. Nat. Acad. Sci. U.S.A. {\bf 83} (1986), no. 10, 3068--3071. 

\bibitem{Bor-90} \bysame, {\em The monster Lie algebra}, Adv. Math. {\bf 83} (1990), no. 1, 30--47.

\bibitem{Bor-92} \bysame, {\em Monstrous moonshine and monstrous Lie superalgebras}, Invent. Math. {\bf 109} (1992), no. 2, 405--444.

\bibitem{Bor-95} \bysame, {\em Automorphic forms on $O_{s+2,2}(\mathbf R)$ and infinite products}, Invent. Math. {\bf 120} (1995), 161--213.

\bibitem{Bor-96} \bysame, {\em The moduli space of Enriques surfaces and the fake Monster Lie superalgebra}, Topology 35 (1996), no. 3, 699--710.

\bibitem{Bor-98} \bysame, {\em Automorphic forms with singularities on Grassmannians},
Invent. Math. {\bf 132} (1998), no. 3, 491--562.

\bibitem{Bru} J. H. Bruinier, {\em On the converse theorem for Borcherds products},  arXiv:1210.4821.

\bibitem{BB} J. H. Bruinier and M. Bundschuh, {\em On Borcherds products associated with lattices of prime discriminant}, Rankin memorial issues, Ramanujan J. {\bf 7} (2003), no. 1-3, 49--61.

\bibitem{Con} J. H. Conway, {\em The automorphism group of the 26-dimensional even unimodular Lorentzian lattice},  J. Algebra {\bf 80} (1983), no. 1, 159--163.

\bibitem{DKN}  T. Damour, A. Kleinschmidt and H. Nicolai, {\em Constraints and the $E_{10}$ coset model}, Classical Quantum Gravity {\bf 24} (2007), no. 23, 6097--6120.



\bibitem{Fein} A. J. Feingold, {\em A hyperbolic GCM Lie algebra and the Fibonacci numbers}, Proc. Amer. Math. Soc. {\bf 80} (1980), no. 3, 379--385.

\bibitem{FF} A. J. Feingold and I. B. Frenkel, {\em A hyperbolic Kac-Moody algebra and the theory of Siegel modular forms of genus $2$}, Math. Ann. {\bf 263} (1983), no. 1, 87--144.

\bibitem{FKN} A. J. Feingold, A. Kleinschmidt and H. Nicolai, {\em Hyperbolic Weyl groups and the four normed division algebras}, J. Algebra {\bf 322} (2009), no. 4, 1295--1339.

\bibitem{FN} A. J. Feingold and H. Nicolai, {\em Subalgebras of hyperbolic Kac-Moody algebras}, Kac-Moody Lie algebras and related topics, 97--114, Contemp. Math., {\bf 343}, Amer. Math. Soc., Providence, RI, 2004.

\bibitem{GN-96} V.~A. Gritsenko and V.~V. Nikulin, {\em Igusa modular forms and ``the simplest" Lorentzian Kac-Moody algebras}, Mat. Sb. {\bf 187} (1996), no. 11, 27--66; translation in Sb. Math. {\bf 187} (1996), no. 11, 1601--1641.

\bibitem{GN-97} \bysame, {\em Siegel automorphic form corrections of some Lorentzian Kac-Moody Lie algebras}, Amer. J. Math. {\bf 119} (1997), no. 1, 181--224.

\bibitem{GN-02} \bysame, {\em On the classification of Lorentzian Kac-Moody algebras}, Russian Math. Surveys {\bf 57} (2002), no. 5, 921--979

\bibitem{Kac} V.~G. Kac,  {\em Infinite-dimensional Lie algebras}, Third edition, Cambridge University Press, Cambridge, 1990.

\bibitem{Kac-97} \bysame, {\em The idea of locality}, arXiv:q-alg/9709008v1.

\bibitem{KMW} V. G. Kac, R. V. Moody and M. Wakimoto, 
{\it On $E_{10}$}, Differential geometrical methods in theoretical physics (Como, 1987), 109--128,
NATO Adv. Sci. Inst. Ser. C Math. Phys. Sci., {\bf 250}, Kluwer Acad. Publ., Dordrecht, 1988.


\bibitem{KaMe}  S.-J. Kang and D. J. Melville, {\em Rank $2$ symmetric hyperbolic Kac-Moody algebras}, Nagoya Math. J. {\bf 140} (1995), 41--75.

\bibitem{KimLee} H. H. Kim and K.-H. Lee, {\em Root multiplicities of hyperbolic Kac-Moody algebras and Fourier coefficients of modular forms}, to appear in Ramanujan J.

\bibitem{KimLee-1} \bysame, {\em Rank $2$ symmetric hyperbolic Kac-Moody algebras and Hilbert modular forms}, preprint, arXiv:1209.1860.

\bibitem{KN} A. Kleinschmidt and H. Nicolai, {\em $E_{10}$ cosmology}, J. High Energy Phys, 2006, no. 1, 137.



\bibitem{LM} J. Lepowsky and R. V. Moody, {\em Hyperbolic Lie algebras and quasiregular cusps on Hilbert modular surfaces}, Math. Ann. {\bf 245} (1979), no. 1, 63--88. 

\bibitem{Nie} P. Niemann, {\em Some generalized Kac-Moody algebras with known root multiplicities},
Mem. Amer. Math. Soc. {\bf 157} (2002), no. 746.

    
\bibitem{Vis} S. Viswanath, {\em Embeddings of hyperbolic Kac-Moody algebras into $E_{10}$}, Lett. Math. Phys. {\bf 83} (2008), no. 2, 139--148.

\end{thebibliography}
\end{document}